\documentclass[11pt]{article}

\usepackage{amsmath,amssymb,amsthm,amsfonts,amstext,amsbsy,amscd}
\usepackage{multirow}
\usepackage{multicol}
\usepackage{latexsym}

\providecommand{\Fourier}{\mathcal{F}}

\usepackage[utf8]{inputenc}
\usepackage{dsfont}
\usepackage{color}
\usepackage[top=3cm, bottom=3.5cm, left=3cm, right=3cm]{geometry}

\renewcommand{\phi}{\varphi}
\DeclareMathOperator{\Log}{Log}

\newcommand{\PP}{\mathbb{P}}
\newcommand{\QQ}{\mathbb{Q}}

\newcommand{\E}{\mathbb{E}}

\newcommand{\R}{\mathbb{R}}
\newcommand{\N}{\mathbb{N}}

\newtheorem{theorem}{Theorem}

\newtheorem{lemma}{Lemma}

\newtheorem{corollary}{Corollary}
\newtheorem{proposition}{Proposition}
\renewcommand{\phi}{\varphi}
\renewcommand{\Log}{\textnormal{Log\,} }
\renewcommand{\d}{\, \textnormal{d}  }
\providecommand{\ebig}[1]{\mathbb{E}\big[ #1  \big]  }
\providecommand{\Ebig}[1]{\mathbb{E}\Big[ #1  \Big]  }

\providecommand{\m}{\textnormal{-} }

\providecommand{\intu}{\int_0^{u} }

\renewcommand{\tilde}{\widetilde}
\providecommand{\intu}{\int_0^u }

\usepackage{mathtools}

\DeclareMathOperator{\lk}{L}

\renewcommand{\hat}{\widehat}
\renewcommand{\Re}{\textnormal{Re}\, }
\renewcommand{\Im}{\textnormal{Im}\,  }
\providecommand{\ephi}{\hat{\phi} }

\providecommand{\dephi}{\hat{\phi}' }
\providecommand{\dphi}{{\phi}' }

\providecommand{\eps}{\varepsilon}

\providecommand{\intm}{\int_{- m}^{ m} }

\DeclareMathOperator{\KL}{KL}

\begin{document}

\title{Nonparametric adaptive estimation for grouped data}
\author{C. Duval\thanks{$^1$MAP5, UMR CNRS
8145, Universit\'e Paris Descartes.},
J. Kappus\thanks{$^2$Institut f\"{u}r Mathematik, Universit\"{a}t Rostock}
}
\date{}

\maketitle

\begin{abstract}
\noindent The aim of this paper is to estimate the  density $f$ of a random variable $X$ when one has access to independent observations of the sum of $K\geq 2$ independent copies of $X$.  We provide a constructive estimator 
based on a suitable definition of the logarithm of the empirical characteristic function.
We propose a new strategy for the data driven choice of the cut-off parameter.  
The adaptive estimator is proven to be minimax-optimal up to some logarithmic loss. A numerical study illustrates the performances of the method.
  Moreover, we discuss the fact that the definition of the estimator applies in a wider context than the one considered here.
  \end{abstract}
\noindent {\sc {\bf Keywords.}} {\small Convolution. Inverse problem. Nonparametric adaptive estimation. } \\
\noindent {\sc {\bf AMS Classification.}} 62G07, 62G20, 62G05.

\section{Introduction}

\subsection{Model and motivation}

In this article, we consider the problem of nonparametric density estimation from  grouped data observations.  The model can be described as follows. Let $X$ be some real-valued random variable which has a square integrable Lebesgue density $f$. We aim at estimating $f$, but we only have access to the aggregated  data
\begin{align}\label{eq data}
Y_j=\sum_{k=1}^{K}X_{j,k},\ \ \ j=1,\ldots,n.
\end{align}
The $X_{j,k},\ k=1,\ldots,K; \,  j=1,\ldots,n$ are independent copies of $X$  and $K$ is a known positive integer.

 This statistical framework is closely related to several other models. First,  we have a connection to decompounding problems (see {e.g.} Buchmann and Gr\"ubel (2003) or van Es {et al.} (2007)), where the jump distribution is estimated from a discretely observed trajectory. In that latter example the law of each nonzero increment is an aggregation of several jumps.

 Second, for $K\geq 2$, \eqref{eq data} describes a particular deconvolution problem where the target density is $f$ and the error density is $f^{\ast K-1}$, which is unknown and depends on $f$ itself. The classical deconvolution model, where the error distribution and the target distribution are not related, has been extensively studied in the literature. Optimal rates of convergence and adaptive procedures are well known. In the case where the density of the noise is known  we can mention  Carroll and Hall (1988), Stefanski (1990), Stefanski and Carroll (1990), Fan (1991), Butucea (2004), Butucea and Tsybakov (2008a,b), Pensky and Vidakovic (1999) or Comte et al. (2007) for $L_{2}$ loss functions or Lounici and Nickl (2011) for the $L_{\infty}$ loss. Deconvolution with unknown error distribution has also been studied (see e.g. Neumann (1997), Delaigle {et al.} (2008), Johannes (2009) or Meister (2009), if an additional error sample is available, or Comte and Lacour (2011), Delattre {et al.} (2012), Johannes and Schwarz (2013), Kappus and Mabon (2014) or Comte and Kappus (2015) under other sets of assumptions). 
 
In the aforementioned literature, it is always assumed that the distribution of the noise is known or can be estimated from an additional sample. In the present framework, one has to exploit the particular structure of the observations and the dependence between the noise and the target density, in order to identify the latter one.  Surprisingly, we will show that the optimal rates of convergence in the grouped data setting with skewed density are, in some cases, slower than the optimal rates which are usually encountered in deconvolution problems (see Theorem \ref{prop_2} below).

\

 Recently, Delaigle and Zhou (2015) studied the nonparametric estimation of the probability of contamination given a covariate $X$, when only aggregated observations of this covariate are available. From an applied point of view, grouped data are encountered in many fields, especially related to the study of infection diseases. For example, aggregated data are encountered when the cost of measurement is too high to allow individual measurements. Moreover,  to ensure confidentiality of the participants in a study, measurements of different individuals can be grouped.  Sometimes, when $X$ represents the measurement of a substance that cannot be detected below a certain detection limit, observations are grouped in order to make the detection possible (see the references given in Delaigle and Zhou (2015)).

\subsection{Main results of the paper}

We have at our disposal $n$  independent copies  of  the random variable $Y$, which has the Lebesgue density $f^{\ast K}$, where $\ast$ denotes the convolution product. This implies that $\phi_{Y}(u)= \phi_X(u)^{K}$, where $\phi_{X}$ and $\phi_{Y}$ denote the characteristic functions of $X$ and $Y$ respectively. Consequently, $f$ can be recovered, taking the $K$-th root of the characteristic function of $\phi_{Y}$  and applying the Fourier inversion formula. A first estimator based on this strategy is proposed in Linton and Whang (2002) to estimate a regression function when only aggregated observations are available. They studied the asymptotic pointwise behavior of their estimator. \\

If $\phi_{X}$ has no imaginary part and is nonnegative, the $K$-th root is uniquely defined and an estimator of $\phi_{X}$ can hence be constructed, replacing $\phi_Y$ by the absolute value of its  empirical counterpart and taking the $K$-th root. An estimator of $f$ is then obtained applying a Fourier inversion formula. Meister (2007) studied this case and provided the associated minimax upper and lower bounds.  Moreover, a cross-validation bandwidth selector is proposed in that paper.  However, theoretical optimality properties for the adaptive procedure have only been  shown under quite restrictive assumptions on the characteristic function. \\
Meister (2007) also proposed a generalization of the  estimation procedure to the skewed case and  a modified version of this estimator is used by Delaigle and Zhou (2015).  However, this estimator has the disadvantage that it is not given in a closed form, making the practical calculation difficult to handle. 

\

In the skewed case,  the estimation procedure suffers from the fact that the definition of $K$-th root is ambiguous. In the present  paper we provide a definition of the $K$-th root, based on a suitable definition of the logarithm (see Lemma \ref{lem Log} hereafter), of a characteristic function that is not necessarily real. A closed form procedure to estimate $f$ naturally follows. For this estimator, we provide non asymptotic risk bounds and derive rates of convergence over certain nonparametric function classes.

It is interesting to note that the smoothness  classes considered in Meister (2007),  which differ from usual Sobolev and H\"older classes, are not very general, since the characteristic functions are bounded from above and below by functions which are of the same order.   In the present work, we provide rates of convergence over more general function classes.  In the particular cases which have been considered in Meister (2007), our rates of convergence agree with the ones found in that paper.  

However, surprisingly, in the general case our rates of convergence are different from the rates which are known to be optimal in a classical deconvolution framework, where the noise density is known or can be estimated from an additional sample.  We prove that these unusual rates are optimal in the grouped data framework (see Theorem \ref{prop_2}). This reveals that the grouped data setting has its own specificities and differ from the classical deconvolution framework. To the best of our knowledge, this phenomenon is completely new in the literature on deconvolution.

\

Finally, we propose a new data driven selection strategy for the cut-off parameter. We show that the resulting rates of convergence are, up to a logarithmic loss,  adaptive minimax. Our adaptive procedure is computationally simple and extremely fast, in comparison to usual cross-validation techniques. Moreover, it applies in a wider setting than the one proposed in Meister (2007).

\subsection{Organization and some notation}

 The paper is organized as follows. Section \ref{sec res} contains  the construction of the  estimator. 
 Upper and lower risk bounds are established and the corresponding minimax rates are derived.  In Section \ref{sec ad} a data driven bandwidth selector is introduced and it is shown that the resulting rates are adaptive minimax. A discussion is proposed in Section \ref{sec discuss} and a numerical study is given in Section \ref{sec num}. Finally, the proofs are gathered in Section \ref{sec proof}.\\

We close this section by introducing  some notation which will be used throughout the rest of the text.  In the sequel, $\ast$ is understood to be the convolution product of integrable functions and $f^{\ast K}$ denotes the $K$-fold auto-convolution, $f^{\ast K}= f\ast \ldots \ast f$. Given a random variable $Z$, 
$\phi_Z(u)=\E[e^{iu Z} ]$ denotes the characteristic function of $Z$. For $Y_j$ defined as in formula \eqref{eq data}, we will drop the subscript and write  $\phi$ instead of $\phi_{{Y_j}}$. For $f\in \lk^1(\R)$, $\Fourier f(u) =\int 
e^{iux } f(x) \d x$ is understood to be the Fourier transform.  Moreover, we denote by $\| \cdot  \|$ the $\lk^2$-norm of functions, $\| f\|^2:= \int |f(x)|^2 \d x$.  Given some function $f\in \lk^1(\R) \cap \lk^2(\R)$, we denote by $f_m$ the  uniquely defined function with Fourier transform $\Fourier f_m= (\Fourier f)\mathds{1}_{[-m,m]}$.

\section{Risk bounds and rates of convergence \label{sec res}}

\subsection{Construction of the estimator}
In the sequel, we will always work under the assumption 
\begin{align}\forall u\in \R, \quad \phi_{X}(u)\ne 0.\tag{A0}\label{eq A0}
\end{align}
This  assumption is standard  in deconvolution problems and, for  the present model,  proven necessary in Meister (2007). Therein, a counterexample is made explicit. When $K$ is even and \eqref{eq A0} is not satisfied, it is possible to build two  densities $f_{0}$ and $f_{1}$ which differ on a set of nonzero Lebesgue measure and for which  $f_0^{\ast K} (x)= f_1^{\ast K}(x),\  \forall x\in\R$. Hence it is not  possible to distinguish  $f_{0}$ from $f_{1}$ in presence of the $K$ aggregated observations.

Suppose, we are given $n$  independent copies  of  the random variable $Y=\sum_{k=1}^{K} X_{1,k}$, which has the Lebesgue density $f^{\ast K}$. In the Fourier domain, this implies $\phi(u)= \phi_X(u)^{K}$. Consequently, $f$ can be recovered, taking the $K$-th root of the characteristic function and using the fact that, by the Fourier inversion formula, 
\begin{align*}
f(x)=\frac1{2\pi}\int_{ \R}e^{-iux}\big(\phi(u)\big)^{1/K}\d u.
\end{align*}
This formula makes immediate sense when $\phi$  has no imaginary part and is strictly positive. Contrarily, if $f$ is non symmetric and hence $\phi$ has a nonzero imaginary part, the definition of  $ z \mapsto z^{1/K}$ is not unique. In order to get a proper definition of the $K$-th root, one needs to specify which branch of the logarithm is considered.  The algorithm  discussed in Meister (2007) for the non symmetric case is not given in a closed form so the estimator is numerically difficult to handle.

In order to get a closed-form estimator  let us recall the concept of the distinguished logarithm. 
\begin{lemma} \label{lem distlog}Let $\phi$ be a characteristic function which has no zeros. Then there exists a unique continuous function $\psi$ which satisfies
\begin{enumerate}
\item[1.] $\psi(0)=0$ 
\item[2.]  $\forall u\in \R: \,  \phi(u)= e^{\psi(u)}$.   
\end{enumerate}
\end{lemma}
 The function $\psi$  introduced in Lemma \ref{lem distlog} is called the \emph{distinguished logarithm} of $\phi$ and, hereafter, denoted by $\Log \phi$. Moreover,  in the sequel, we let $1/K$ denote the 
 \emph{distinguished $K$-th} root, that is, $\phi^{1/K}:= \exp( ( \Log \phi)/K  )$.  It is important to keep in mind that the distinguished logarithm is usually not equal to the main branch of the logarithm. Moreover, the definition of a distinguished logarithm, and consequently of a distinguished root, only makes sense with respect to $\phi$ seen as a function rather than pointwise. 

\begin{lemma}\label{lem droot} For some integer $K$, let  $\phi(u)=\phi_X(u)^K$. Then, both characteristic functions are related as follows, 
\begin{align*}
\forall u\in \R:   \phi_X(u) = \exp(   (\Log \phi (u) )/K )=: \phi(u)^{\frac{1}{K}}. 
\end{align*}
\end{lemma}
The proof of Lemma \ref{lem distlog} and Lemma \ref{lem droot} can be found, for example, in Sato  (1999) (Lemma 7.6). 

In order to obtain  a constructive estimator for $\phi_X$, we give the following result which makes the definition of the distinguished logarithm explicit.  

\begin{lemma}\label{lem Log} Let $\phi$ be a  characteristic function without zeroes and assume that $\phi$ is differentiable.  Then, it follows that 
\begin{equation}  \label{Formel_DL}
\Log (\phi(u) ) = \int_0^u  \frac{\phi'(z)}{\phi(z) } \d z.
\end{equation}
\end{lemma}

The preceding lemmas suggest to  exploit formula \eqref{Formel_DL} in order to build  a constructive estimator of the distinguished logarithm and hence of the characteristic function itself.  The characteristic function of $Y$ and its derivative are replaced by their empirical counterparts,  
$$\hat\phi(u):=\frac1n\sum_{j=1}^{n}e^{iuY_{j}}   \quad \text{and}  \quad \hat{\phi}'(u)= \frac{1}{n}\sum_{j=1}^{n}  iY_j e^{iu Y_j}. $$ 

Moreover, we define the quantities 
\[
 \psi'(u) :=\frac{\phi'(u)}{\phi(u) }\quad\mbox{and}\quad\hat{\psi}'(x):=\frac{\ephi'(u) }{\ephi(u) }.
\] 
as well as 
\[
\psi(u) := \intu \psi'(x) \d x \quad\mbox{and}\quad  \hat{\psi}(u):= \intu \hat{\psi}'(x)\d x. 
\]
It follows from Lemma \ref{lem Log} that   $\psi(u) =\Log (\phi(u) )$ and $\hat{\psi}(u) =\Log (\hat{\phi}(u) )$. Lemma \ref{lem droot} suggests to use 
\begin{align*}
\hat{\phi}_X(u):=\exp(  \hat{\psi}(u) /K )
\end{align*} 
as an estimator of  $\phi_X$. Finally, we  use a spectral cut-off and  apply Fourier inversion to  derive, for arbitrary $m>0$,  the following 
estimator of $f$, 
\begin{align}
\hat f_{m}(x):=\frac{1}{2\pi}\int_{-m}^{ m}e^{-iux}\hat\phi_{X}(u)\d u.\label{eq def est}
\end{align}

It is worth pointing out that for real-valued and strictly positive characteristic functions, our estimator coincides with the estimator given in Meister (2007). In this case, the distinguished $K$-th root equals the usual $K$-th root. 
This estimator resembles one of the estimators defined in Comte et al. (2015), who consider the problem of estimating the jump density for mixed compound Poisson processes. However, in Comte et al. (2015) the associated upper bound is sub optimal and an adaptive bandwidth selector is not proposed. 
\subsection{Non asymptotic risk bounds} 
For  $\eps, \gamma>0$, define 
\begin{align}\label{eq uek}
u_n^{ (\gamma,\eps)}:= \min\{u\geq 0: |\phi(u)|=(1+\eps) \gamma( n/\log n)^{-1/2}\}. 
 \end{align}

 The following bound is valid for the mean integrated squared error of $\hat{f}_m$. 

\begin{theorem}  \label{Hauptsatz_Schranke} Assume that $X$ has a finite second moment. Let $\gamma:=\sqrt{1+ \frac{2}{K}+\delta}$, for some  $ \delta>0$.
Then, for any  $m\geq 1$, 
\begin{equation} \label{equation_upper}
\E[\| f- \hat{f}_m\|^2]  \leq \| f -f_m\|^2  +\frac{C_1}{nK^2}  \int_{-m }^{m } \limits  \frac{1}{|\phi_X(u)|^{2(K-1)}}  \d u+ C_2 m  n^{-\frac{1}{K}} + \frac{2}{\pi} \hspace*{-0.5cm}\int_{ \{ |u| \in U_{n}^{\gamma,\eps} \} }\limits \hspace*{-0.3cm} |\phi_X(u) |^2 \d u,
\end{equation}
where $U_{n}^{\gamma,\eps}=[u_{n}^{(\gamma,\epsilon)} ,m]$, with the convention that if $m\leq u_{n}^{(\gamma,\epsilon)}$ then $U_{n}^{\gamma,\eps}=\emptyset$, and where $C_1=(1+\eps)\log(1+1/\eps)$ and  $C_2=(9+4\|\phi_{X}\|^{2})$.
\end{theorem}
\  \\[-0.2cm]
 The first two terms appearing in the upper risk bound illustrate the structural similarity between this grouped data setting and a classical density deconvolution framework.  The density $f$ of $X$  plays the role of the density of interest and $f^{\ast (K-1)} $  resembles the density of the error term. The third summand corresponds to the usual  variance term for density estimation from grouped data when the density is symmetric (see Meister (2007)). Finally, the fourth summand depends on the quantity $u_{n}^{(\gamma,\epsilon)}$ which appears naturally. We prove that on the interval $[- u_{n}^{(\gamma,\epsilon)},u_{n}^{(\gamma,\epsilon)}]$, the estimator $\hat{\phi}$ can be controlled uniformly. This allows to control the deviation of the empirical distinguished logarithm and consequently the deviation 
of $ \hat{\phi}_X$ from~$\phi_X$.   \\

It follows directly  from the proof of Theorem \ref{Hauptsatz_Schranke} (see Section \ref{sec prf1}) that in the symmetric case the upper bound simplifies, for any $m>0$, to 
 $$\E[\| f- \hat{f}_m\|^2]  \leq \| f -f_m\|^2  + C m  n^{-\frac{1}{K}}.$$
The second and fourth term appearing on the right hand side of formula \eqref{equation_upper} are particular to the case where the density is skewed. Balancing those four terms, there will occur  a logarithmic loss, compared to the symmetric case. This phenomenon is  a consequence of the fact that the estimation problem in the non-symmetric case involves the distinguished logarithm, which requires uniform rather than pointwise control of the empirical characteristic function. Controlling suprema typically leads to the loss of a logarithmic factor. However, we do not have a lower bound result that establishes whether this logarithmic loss is avoidable or not.

\subsection{Rates of convergence \label{sec rate}}
Let us investigate the rates of convergence over certain nonparametric classes of densities. In the sequel, for $b\geq \beta >\frac{1}{2}$,  we denote by $\mathcal{F}(\beta, b,  C, C', C_X, C_f)$ the class of densities for which the following holds
\begin{align}\label{eq F}
\forall u \in \R: \,   (1+C|u|)^{-b}  \leq |\phi_{X}(u) |  \leq (1+C'|u|)^{-\beta}
\end{align} 
and, in addition,
\begin{equation}  \label{Konst_Schranke}
\E[ X^{2}] \leq C_X   \quad  \text{and}   \quad \|f\|\leq C_f. 
 \end{equation}
For $\rho>0$, $\beta>\frac{1}{2}$, $b\in \R$ and $c>0$,  let $\mathcal{H}(\beta, \rho, b,c,  C, C', C_X, C_f)$ be the class of densities for which \eqref{Konst_Schranke} holds and, in addition, 
\begin{align*}
\forall u \in \R: \,   (1+C|u|)^{-b} \exp(-c|u|^{\rho} ) \leq |\phi_{X}(u) |  \leq (1+C'|u|)^{-\beta}.
\end{align*}

Finally, we denote by   $\mathcal{G}(\rho, c, C,C',C_X,C_f) $ the class of densities for which  \eqref{Konst_Schranke} holds and, in addition,
\begin{align}\label{eq G}
\forall u\in \R:   (1+C |u|)^{(\rho-1)/2 } \exp( - c |u|^{\rho}  ) \leq |\phi_{X}(u)|  \leq (1+C' |u|)^{(\rho-1)/2} \exp( - c |u|^{\rho}  ) .
\end{align}

In what follows,  for the sake of readability, we drop the dependence on $C, C', C_X$ and $C_f$ when there  is no risk of confusion. We write, for example $\mathcal{F}(\beta, b)$ instead of $\mathcal{F}(\beta, b,  C, C', C_X, C_f)$.
 As it is usually employed in the  literature, we say that a density is ``ordinary smooth'' if its characteristic function decays polynomially as in \eqref{eq F} and that it is ``super smooth'' if its characteristic function decays exponentially as in \eqref{eq G}.

\

The  classes of functions introduced above differ from the Sobolev or H\"older classes which are usually considered in density estimation. These types of function classes are characterized by some bound from above on the decay of the characteristic function. But contrary to the present model,  bounds from below are not imposed. 
 However, the statistical model under consideration is a special case of a deconvolution problem, with target density $f$ and noise density $f^{\ast (K-1)}$. In a deconvolution framework,  rates of convergence over nonparametric classes of target densities can only be established under additional assumptions on the decay of the characteristic function of the noise. This explains why we consider classes of functions that provide a control from above and below of the characteristic function. It permits to render explicit the rates of convergence, and particular to be able to handle the second term in \eqref{equation_upper} and the quantity $u_{n}^{\gamma,\eps}$ defined in \eqref{eq uek}.

\

The following result is an immediate consequence of Theorem \ref{Hauptsatz_Schranke}.
\begin{proposition}  \label{prop}
For  $a>0$, set  
\[
u_n^{(a)} := \min\{u \geq 0:  |\phi(u) |  = a n^{-\frac{1}{2} }  \}.
\]
Fix $a,\eps, \delta >0$ and let $\gamma$ be defined as in Theorem \ref{Hauptsatz_Schranke}. For every $n\in \N$, let $m^*\in [u_{n}^{(\gamma,\epsilon)}, u_n^{(a)}]$. Then, we have 

\begin{enumerate}\item[(i)]  $
\underset{{ f\in \mathcal{F}(\beta, b)  }}{\sup}  \E[ \|f - \hat{f}_{m^*}\|^2  ] =  O \Big( (n/\log n) ^{-\frac{(2\beta - 1)}{2 K b}  } \Big). 
$
\item[(ii)] $ \underset{{ f\in \mathcal{H}(\beta,\rho,b, c)  }}{\sup}  \E[ \|f - \hat{f}_{m^*}\|^2  ] =  O \Big( (\log n)^{\frac{2\beta - 1}{\rho} } \Big). 
$
\item[(iii)]  $\underset{ f\in \mathcal{G}(\rho,c)}{
\sup}    \E_f[ \|f - \hat{f}_{m^*}\|^2  ] =  O \Big( n^{-\frac{1}{K}  } ( \log n)^{\frac{1}{\rho} }  \Big). $
\end{enumerate}
\end{proposition}

If we consider the class $\mathcal{F}(\beta, \beta)$  (taking $b=\beta$) or the class $\mathcal{G}(\rho, c)$, the rates of convergence summarized in  Proposition 1 coincide with the rates which have been proven to be minimax optimal (up to a logarithmic loss) in Meister (2007).  Moreover, the rates of convergence derived for  $\mathcal{H}(\beta, \rho,b,c)$  coincide with the optimal rates of convergence encountered in a deconvolution setting when the target density is ordinary smooth and the density of the noise is super smooth, see e.g. Fan (1991).
\   \\

The general  class $\mathcal{F}(\beta, b)$, with $b>\beta$, has not been investigated in Meister (2007).  It is unexpected and noteworthy that the rate of convergence derived in Proposition \ref{prop} differs from the optimal rate of convergence which is "standard" in  deconvolution problems. 

Indeed, take the case where a random variable $X$ is observed with an additional additive error $\epsilon$ whose distribution is  known.  If $|\phi_X|$ is bounded from above by a function which decays as $|u|^{-\beta}$ and  $|\phi_\epsilon|$ is bounded from below by a function decaying as $|u|^{-(K-1)b}$, the optimal rate of convergence is $n^{-\frac{2\beta -1}{2\beta +2(K-1)b} }$ (see e.g. Meister  (2009)). If the characteristic function of $\epsilon$ is unknown, but can be estimated from an additional sample of size $n$ of the pure noise, the rate of convergence is the same as in the case of a  known error distribution (see e.g. Johannes~(2009)).

In comparison to this, the rate of convergence derived in  Proposition \ref{prop} equals, up to a logarithmic factor  $n^{\frac{2\beta-1}{2bK} }$, which is slower.
It is shown in the next section that this slower rate of convergence is not due to our estimator, but is, up to the logarithmic factor, optimal in the grouped data framework. 

\

We have not discussed  the case where $|\phi_{X}|$ is bounded from above and below by exponentially decaying functions with different orders. It is possible to derive rates of convergence in this case. However, computations are cumbersome and the resulting rates are, in some particular cases, not available in a closed form (see Lacour (2006)). For this reason, we omit this case. 

\subsection{Minimax lower bound}
The following result shows that, if $\beta>1$, the rates of convergence found for the class $\mathcal{F}(\beta, b)$ are minimax optimal up to the loss of a logarithmic factor. 

\begin{theorem}\label{prop_2}  Assume that $\beta>1$. Then there  exists some positive constant $d$ such that 
\begin{align*}
\liminf_{n\to \infty} \     \inf_{\tilde{f}  }  \  \sup_{f\in \mathcal{F}(\beta, b) }
 \E_f[ \|f - \tilde{f}\|^2  ] n^{\frac{ 2\beta - 1}{2bK}  }  \geq d >0.  
\end{align*}
The infimum is taken over  all estimators $\tilde{f}$ of $f$,  based on the observations $Y_j, \ j=1, \ldots, n$. 
\end{theorem}
Theorem \ref{prop_2} does not cover the cases where $\beta \in (0.5,1]$. However, at the expense of some additional technicalities, it is possible to extend the result to this case. We neither formulate nor prove lower bound results associated to the upper bounds of Proposition \ref{prop}, part (ii) and (iii). Indeed, part (ii) matches the optimal convergence rate for density deconvolution problems
with known error density  and part (iii) has been proven optimal in Meister (2007). \\

Gathering the results of Proposition \ref{prop} and Theorem \ref{prop_2} we have established that our estimation procedure is rate optimal up to a logarithmic factor. It remains unclear if the logarithmic loss  is avoidable. However, we suspect that the rate cannot be improved as some uniform control of the empirical characteristic function is required, which typically leads to a logarithmic loss.

\section{Adaptive estimation}\label{sec ad}

The cut-off $m^*$ is defined in terms of the characteristic function:  the estimator of $\hat{\phi}_X$ is set to zero as soon as $\phi(u)$ is below some critical threshold which is essentially of the order $n^{-1/2}$.  This is intuitive, since a reasonable estimate of the characteristic function in the denominator is no longer possible when this object is of smaller order than the standard deviation.  
It is interesting to notice that this theoretical cut-off $m^*$, which may vary in a certain interval and guarantees rate optimality of the estimator, has an empirically accessible counterpart.  
This motivates the definition of the empirical cut-off in terms of the empirical characteristic function. Once $\hat{\phi}$ is (up to an additional logarithmic term) below some critical threshold, $\hat{\phi}_X$ is set to ~zero. For  a numerical  constant $\eta>1$, to be chosen, we  introduce the data driven version of $m^*$,
\begin{align*}
\hat{m}_{\eta}:=\min\Big\{  \min\{u: |\hat{\phi}(u) | \leq  (K n)^{-\frac{1}{2} } + \sqrt{ \eta /K} (n/\log n)^{-\frac{1}{2}  } \} , n^{\frac{1}{K} } \Big\}. 
\end{align*}
The corresponding  estimator $\hat{f}_{\hat{m}_{\eta} }$ adapts automatically to the unknown smoothness class, so  the convergence rate is simultaneously minimax over a collection of nonparametric classes. 
\begin{theorem}\label{Hauptsatz_ad}Let constants $\overline{B}, \underline{B}>\frac{1}{2}, \overline{C},\underline{C'}>0, \overline{C}_X$ and $\overline{C}_f$ be given.
\begin{itemize}
\item[(i)]    Define
\[
{I}:= \{(\beta,b):  \overline{B}\geq b \geq \beta >\underline{B}\} \times (0, \overline{C}]\times [\underline{C'},\infty) \times (0, \overline{C}_X]  \times (0,  \overline{C}_f].
\]
 Then, there exists some positive real $\mathcal{C}$ such that 
\begin{align*}
\sup_{ (\beta, b,C, C',C_X,C_f)  \in  {I}  } \Big( \frac{ \sup_{ f\in \mathcal{F}(\beta, b,C, C',C_X,C_f)  }  \E_f[ \|f - \hat{f}_{\hat{m}_{\eta} }\|^2  ]  }{   (n/\log n) ^{-\frac{(2\beta - 1)}{2 K b}  }        }\Big)  
\leq \mathcal{C}. 
\end{align*}
\item[(ii)]  Define   
\begin{align*}
{I}:= [\underline{B},\infty) \times (0,\overline{\rho} ]  \times(-\infty, \overline{B} ] \times (0, \overline{c}  ] \times (0,\overline{C} ] \times [\underline{C'},\infty)  \times (0, \overline{C}_X]  \times (0,  \overline{C}_f]
.
\end{align*}
Then, there exists a positive real $\mathcal{C}$ such that 
\begin{align*}
\sup_{(\beta, \rho, b,c,  C, C', C_X, C_f)\in I  }\Big( \frac{ \sup_{ f\in \mathcal{H}(\beta,\rho,b, c, C, C',C_X,C_f)  }  \E_f[ \|f - \hat{f}_{\hat{m}_{\eta} }\|^2  ]  }{   (\log n) ^{-\frac{(2\beta - 1)}{\rho}  }        }\Big)  \leq \mathcal{C}.
\end{align*}
\item[(iii)] Let
\[
I= (0,\overline{\rho}]\times (0,\overline{c} ]\times  [\underline{C}, \overline{C}]\times [\underline{C'}, \overline{C'}  ]\times (0,\overline{C}_X]\times (0, \overline{C}_f].
\]
Then, there exists some positive real $\mathcal{C}$ such that 
\begin{align*}
\sup_{ (\rho,c, C, C',C_X,C_f)  \in  {I}  }\Big( \frac{\sup_{ f\in \mathcal{G}(\rho,c, C, C',C_X,C_f)  }  \E_f[ \|f - \hat{f}_{\hat m_{\eta}}\|^2  ] }{n^{-\frac{1}{K}  } ( \log n)^{\frac{1}{\rho} }    } \Big) 
\leq \mathcal{C}. 
\end{align*}
\end{itemize}
\end{theorem}

 In earlier works, the adaptive bandwidth selection has only been investigated for the particular case of  symmetric densities and for polynomially decaying characteristic functions. Moreover, the cross-validation strategy proposed in Meister (2007) cannot be extended to the present framework, since it leads to the choice $n^{-1/(2\beta K)}$, which gives a sub-optimal rate of convergence for $b>\beta$. Finally, in comparison to cross validation, our adaptive estimator has the strong advantage of being computationally extremely simple and fast, since it only requires the evaluation of the characteristic function. \\

The adaptive choice $\hat m_{\eta}$ depends on the constant $\eta$ that needs to be chosen. I is not unusual when it comes to adaptation that an hyper parameter needs to be calibrated. The same holds true in  the context of wavelet adaptive estimation or when the spectral cut-off is obtained by minimization of a penalty.  Our simulation results indicate that the procedure works well if $\eta$ is chosen close to $1$.

\section{Discussion \label{sec discuss}}

In this article, we provided a  fully constructive adaptive and optimal estimation procedure in the context of grouped data observations, which covers both the non symmetric and symmetric cases.  The study of this framework revealed that it has its own specificities and differs from what one would expect in a classical deconvolution framework.  

It is also worth noticing that if $K$ increases, the rates of convergence obtained deteriorate in Theorem \ref{Hauptsatz_ad}. If we let $K\rightarrow \infty$ then,  (see Meister (2007) for the lower bound results)
$$\underset{\hat f}{\inf}\underset{f\in \mathcal{F}}{\sup}\E\big(\|\hat{f}_m-f\|^2\big|\big)>0,$$ where $\mathcal{F}$ is a class of densities, for instance one of those defined in Section \ref{sec rate}. This result is intuitive. Assume that the density $f$ has expectation $\mu$ and finite variance $\sigma^{2}$. Then we have $$\frac{(Y_{j}-\mu)}{\sqrt{K}}\underset{K\rightarrow \infty}{\longrightarrow}\mathcal{N}(0,\sigma^{2}).$$ It means that for large $K$, each observation $Y_{j}$ is close in law to a Gaussian random variable depending on two parameters only. The whole density $f$ is then lost, only its expectation and variance may be recovered from the observations.  This phenomenon is also observed in Duval (2014) in the case of a compound Poisson process observed at a sampling rate tending to infinity. \\

Finally, we want to emphasize that the adaptive estimator we define in this paper applies for a wider range of problems than grouped data. In fact, it can be used each time the quantity of interest can be recovered through a closed form inverse of the characteristic function of the observations.

Suppose that a L\'evy process  $Y$ is observed over $[0,T]$ at low frequency and denote by $\Delta\geq 1$ the sampling rate.  A vast literature is available on the estimation of the underlying parameters, and in particular the estimation of the L\'evy measure (see among others Neumann and Rei\ss~(2009), Comte and Genon-Catalot (2010) or Kappus~(2014)).
However, one may also consider the case where the quantity of interest is the distributional density of $Y_t$ for some $t>0$ (without loss of generality set $t=1$) rather than the jump density.

Suppose that $Y_1$ has a square integrable Lebesgue density $f$.  Let $\phi_1$ and $\phi_\Delta$ be the characteristic functions of $Y_1$ and $Y_\Delta$, respectively.  Then both characteristic functions are connected as $\phi_\Delta(u) = (\phi_1(u) )^{\Delta}$ or equivalently as $\phi_1 = \exp( (\Log\phi_\Delta)/\Delta )$. The estimation procedure proposed in this paper carries over immediately, leading to the estimator 
\[
\hat{f}_m(x) = \frac{1}{2\pi} \intm e^{ - iu x  } \hat{\phi}_1(u)   \d  u, \quad \text{with}  \quad  \hat{\phi}_{1} (u) = \exp\Big(\frac{1}{\Delta}\intu \frac{\hat{\phi}'_\Delta(v) }{\hat{\phi}_\Delta(v) }  \d v\Big)
\]
and the upper bound 
\begin{align*}
\E[  \| f- \hat{f}_m \|^2]  \leq \| f- f_m  \|^2 +\frac{C_1}{\Delta^2 n  } \intm   \frac{ |\phi_1(u) |^2}{ |\phi_{\Delta}(u) |^2  }  \d u  + C_2  \int_{ [m, u_{n}^{(\gamma,\epsilon)} ]  } \limits |\phi_1(u)|^2 \d u+ C_3 n^{-\frac{1}{\Delta} }.
\end{align*}
Balancing the terms when $\phi_{1}$ is polynomially decreasing leads to the rate  $(T/\log T)^{-\frac{2 \beta - 1}{2\Delta \beta}  }$.  For exponentially decaying characteristic functions, the convergence rate is, up to a logarithmic term, $n^{-\frac{1}{\Delta} }$. In a sense, the context of infinitely divisible distributions can be interpreted as a "grouped data model" with a non-integer group size $\Delta$. The choice of $\widehat m_{\eta}$ proposed in the present publication will lead to a minimax adaptive estimator.

\section{Numerical results} \label{sec num}

In this section we illustrate the behavior of our adaptive procedure by computing $L_{2}$ risks for various density functions $f$, when $n$ and $K$ are varying. The upper bound of Theorem \ref{Hauptsatz_Schranke} states that these risks should decrease as $n$ increases and should increase with $K$. As points of comparison, we compute also the risks of two different procedures. First, the cross-validation  estimator proposed in Meister (2007), even though no theoretical justification has been established that the cross-validation strategy works outside the particular case where $f$ is symmetric and the characteristic function decays polynomially. Second,  an oracle "estimator". More precisely the estimator $\widehat f_{m^{\star}}$, which is the estimator defined in \eqref{eq def est} where $m^{\star}$ corresponds to the following oracle bandwidth
\begin{align*}
m^\star= \arg\min_{m>0 } \|  f- \hat{f}_{m}  \|^2.
 \end{align*} 
This oracle can be explicitly computed if $f$ is known. We denote these different risks by $r,$ for the risk of our procedure, $r_{cv}$ for the cross-validation estimator  and $r^{*}_{or}$ for the oracle procedure. All three estimators are computed on the same samples.
 
 The $L^{2}$ risks are computed after 500 Monte-Carlo simulations of each estimator and averaging the results.  We consider the following group size  $K=5,\ 10,\ 20$ and $50$ and the following distributions:

 \begin{enumerate}
 \item [(i)] Normal distribution with mean value $2$ and variance $1$.
 \item [(ii)]  Gumbel distribution with mean $3$ and scaling parameter $1$. 
 \item [(iii)] Gamma distribution with parameters $6$ and $3$. 
 \item[(iv)]  Laplace distribution with mean $0.5$ and scale parameter $\beta=3$. 
 \end{enumerate}

 Moreover, we set $\eta =1.1$ for the adaptive procedure. Numerical experiments indicate that the procedure is sensitive to the choice of the numerical constant $\eta$ and that this constant should be chosen close to $1$.  The results are summarized in the tables below. 

\

\begin{center}

\begin{tabular}{|c|c||c|c|c||c|c|c|}
\hline 
    &     &       \multicolumn{3}{c||}{  $\mathcal{N}(2,1)$ }   &       \multicolumn{3}{c|}{ Gumbel(3,1) }    \\
\hline
$n$    &    $ K $    &   $ r_{or}^{*}$    &  $ r $    &  $ r_{cv} $    & $ r_{or}^{*}$    &  $ r $    &  $ r_{cv} $       \\
\hline 
\hline 
\multirow{4}{*}{1000}& 5 &  0.033 &  0.088& 0.095&  0.017    &  0.037    &    0.045    \\ 
\cline{2-8}
 &  10 & 0.045  & 0.124        &  0.144     &   0.031
   & 0.066    &     0.082     \\ 
\cline{2-8}
 & 20 & 0.057   &  0.156      &    0.185  & 0.043    &   0.092   &    0.111         \\
\cline{2-8}
 & 50 & 0.069  &    0.185    &   0.236    &   0.053
   &  0.117   &  0.154          \\  
\hline 
\multirow{4}{*}{5000}& 5 &  0.031   & 0.074      &    0.073  &   0.012   &  0.027    &   0.035      \\
\cline{2-8}
 &  10 & 0.039  &  0.114      &   0.123     &      0.027 
   &  0.056   &     0.062           \\
\cline{2-8}
 & 20 &0.053   &     0.147   & 0.168& 0.040 &  0.083 & 0.101     \\
\cline{2-8}
 & 50 & 0.066  &   0.183     &  0.227    &  0.052 & 0.110  &  0.143
         \\
 \hline 
\multirow{4}{*}{10000}& 5 & 0.007 & 0.018  &  0.021    &   0.011 &  0.025 &  0.038      \\
\cline{2-8}
 &  10 &0.017 & 0.045 &  0.055    &      0.025  &  0.053 & 0.059     \\
\cline{2-8}
 & 20 &    0.028  & 0.077 &  0.099    &   0.039 & 0.081 & 0.094 \\
\cline{2-8}
 & 50 &   0.039  & 0.105  &  0.221    &    0.051 & 0.114 & 0.140      \\
 \hline 
 \end{tabular}

\vspace{0.7cm}

 \begin{tabular}{|c|c||c|c|c||c|c|c|}
  \hline 
     &     &       \multicolumn{3}{c||}{  $\Gamma(6,3)$ }   &       \multicolumn{3}{c|}{Laplace(0.5,3) }    \\
\hline
$n$    &     $K$     &   $ r_{or}^{*}$    &  $ r$    &  $ r_{cv} $    & $ r_{or}^{*}$    &  $ r $    &  $ r_{cv} $       \\
\hline 
\hline 
\multirow{4}{*}{1000}& 5 &     0.021 &  0.050  & 0.060   & 0.070 &  0.152 & 0.149  \\
\cline{2-8}
 &  10 &0.039 & 0.089 &  0.110     & 0.114 &  0.239 &  0.260        \\
\cline{2-8}
 & 20 &0.054 & 0.125 & 0.158  &   0.150 &  0.312 &  0.367   \\
\cline{2-8}
 & 50 & 0.067  &  0.162  & 0.212     &     0.180 &  0.372 &  0.468
   \\
\hline 
\multirow{4}{*}{5000}& 5 & 0.016 & 0.037  & 0.050  &  0.055 &  0.118  & 0.121
  \\
\cline{2-8}
 &  10 &  0.033 & 0.076  &  0.090       &    0.100   & 0.214  & 0.231
    \\
\cline{2-8}
 & 20 &   0.049 &  0.116 & 0.134    & 0.140 & 0.294 & 0.324
     \\
\cline{2-8}
 & 50 &  0.065 & 0.159 & 0.202   & 0.173 & 0.374 &  0.438
   \\
 \hline 
\multirow{4}{*}{10000}& 5 &   0.013 & 0.032 & 0.040    &     0.047  &  0.103 & 0.098 \\
\cline{2-8}
 &  10 &0.030  &  0.071 &  0.090   &  0.094  &  0.197 &  0.208
  \\
\cline{2-8}
 & 20 &0.047 & 0.111 & 0.131   &  0.134  &  0.287 & 0.319
      \\
\cline{2-8}
 & 50 &  0.064   &  0.153  &  0.197  &    0.170  &  0.350  &  0.439    \\
 \hline 
\end{tabular}
\   \\[0.2cm]

\end{center}

From the above tables, we observe that all procedures performs well in all examples. As expected, we observe that for all procedures, the associated risks decrease as $n$ increases and that increasing $K$ deteriorates the risks. Moreover the adaptive strategy proposed in the present article shows a slightly better outcome than the cross-validation techniques. It is also numerically extremely efficient since it requires only one evaluation of the empirical characteristic function.

\section{Proofs}\label{sec proof}
\subsection{Preliminary results}

\begin{lemma}\label{An_Lemma2} Let $Z$ be an integrable random variable with characteristic function $\phi$.  Then, for  any  $h>0$,  
\[
\forall u\in \R:   |\phi(u)- \phi(u+h)|  \leq h  \E[|Z|]. 
\]
\end{lemma}
\begin{proof} Using that for any  $x,y\in \R$, $|e^{ix}- e^{iy}|\leq |x-y|$, leads to the result
\begin{align*}
&|\phi(u) - \phi(u+h)|  = |\E[e^{i uZ} - e^{i (u+h)Z} ]  |   \leq  h\E[|Z|].
\end{align*}
\end{proof}
For arbitrary  $c>0$, define the event  
\[
A_{c}:=\Big\{  \Big| \frac{1}{n}\sum_{j=1}^{n}  \big( |Y_j| -\E[|Y|] \big) \Big| \leq  c\Big\} .
\]
By the Markov inequality,   we derive that
\[
\PP(A_c^c) \leq \frac{1}{c^2 n} \E\big[Y^2\big] . 
\] 
\begin{lemma} \label{An_Lemma1} Let $h>0$. Then we have on the event $A_{c}$ the following inequality
\[
\forall u \in \R:  |\hat{\phi}(u) -\hat{\phi}(u +h)|  \leq h  \big(\E[|Y|]+c  \big). 
\]
\end{lemma}
\begin{proof} The triangle inequality and the fact  that $x\rightarrow e^{iux}$ is Lipschitz with constant 1, give
\begin{align*}  
   |\hat{\phi} (u) - \hat{\phi}(u+h) \big|    \leq    \frac{h}{n}  \sum_{j=1}^{n} |Y_j|. 
\end{align*}
The definition of $A_c$ leads to
\begin{align*}
  \frac{h}{n}  \sum_{j=1}^{n} |Y_j|   = h \Big(\E[|Y|] +   \frac{1}{n}\sum_{j=1}^{n} (|Y_j|- \E[|Y|] ) \Big)\leq  h (\E[|Y|]+c).
\end{align*} 
\end{proof}

\begin{lemma} \label{An_Lemma3}Fix $h>0$ and some  $u>0$. Define the grid points $t_k:=kh,\ k\in\{1,  \ldots \lceil \frac{u}{h} \rceil\}$. Then the following holds true for arbitrary $\tau>0$.  
\begin{align*}
\PP\Big(\exists  k\in \{1,\ldots, \lceil \tfrac{u}{h} \rceil\}: |\hat{\phi}(t_k ) - \phi(t_k) | > \tau \Big( \frac{  \log n }{n}\Big)^{\frac{1}{2} }   \Big) \leq  2\lceil \tfrac{u}{h} \rceil  n^{-\tau^2/2 }.
\end{align*}
\end{lemma}
\begin{proof}
This  is a consequence of the series of inequalities 
\begin{align*}
&\PP\Big(\exists k\leq \lceil \tfrac{u}{h} \rceil :  | \hat{\phi} (t_k) - \phi(t_k)        |   \geq     \tau \Big(  \frac{\log n}{n} \Big)^{\frac{1}{2} }   \Big)   
 \leq   \sum_{k=1}^{\lceil u/h \rceil} \PP \Big(  | \hat{\phi} (t_k) - \phi(t_k)        |   \geq     \tau \Big(  \frac{\log n}{n} \Big)^{\frac{1}{2} }  \Big)  \\
& \leq   \sum_{k=1}^{\lceil u/h \rceil} \PP \Big(  \sum_{j=1}^{n}|e^{it_{k}Y_{j}} -\E[e^{it_{k}Y_{j}}]       |   \geq     \tau \big( n {\log n} \big)^{\frac{1}{2} }  \Big)  \leq   \sum_{k=1}^{\lceil u/h \rceil}  2\exp( - \tau^2 (\log n)  )   =2\lceil u/h\rceil n^{-\frac{\tau^{2}}{2} }. 
\end{align*} 
The  last inequality follows from the Hoeffding inequality.\end{proof}

Fix $ c, h$ and $\tau$. 
For arbitrary $u>0$, define the event 
\[
B_{c,h,\tau}(u):= A_{c}  \cap \Big\{  \forall k\leq  \lceil \tfrac{u}{h}\rceil  :    |\hat{\phi}(t_k)- \phi(t_k) |  \leq 
\tau \Big( \frac{\log n}{n}    \Big)^{1/2}  \Big\} .
\]
\begin{lemma}\label{An_Lemma4} On the event $B_{ c,h,\tau}(u)$, the following holds
\[
\sup_{v\in [0,u ]  }  |\hat{\phi}(v) - \phi(v)|   \leq h (2 \E[|Y|]  +c )  + \tau \Big( \frac{\log  n}{n} \Big)^{1/2}. 
\]
\end{lemma}
\begin{proof}  Consider the grid point $t_k$ defined in Lemma \ref{An_Lemma3} and $v \in [t_k-h/2,t_{k} +h/2]$ for some $k\leq \lceil \tfrac{u}{h}\rceil$.  Then,  applying Lemma \ref{An_Lemma2}  and Lemma \ref{An_Lemma1}, we have on the event $B_{c,h,\tau}(u)$
\begin{align*}
 |\hat{\phi}(v) - \phi(v)|  \leq & | \hat{\phi}(v) - \hat{\phi}(t_k) |  +| \hat{\phi}(t_k) - \phi(t_k) |
+  |\phi(t_k)- \phi(v)|\\
\leq & h(2\E[|Y|]  +c )  + \tau \Big( \frac{\log  n}{n} \Big)^{1/2} . 
\end{align*}
This is the desired result. 
\end{proof}

The following Corollary \ref{Korollar} is a consequence of Lemma \ref{An_Lemma4}.

\begin{corollary}\label{Korollar} Assume that  $h=o((\log n/n)^{1/2})$. Then, for arbitrary $\tau<\kappa$, there exists some positive integer  $n_0 $, that does not depend of $u$, such that  for any $n\geq n_0$,  on the event $B_{c,h,\tau}(u)$ it holds that
\begin{align*}
\forall  v\in [0,u]: \, |\hat{\phi}_n (v) - \phi(v)|  \leq \kappa (\log n/n)^{\frac{1}{2} }. 
\end{align*}
\end{corollary} Note that  this integer $n_{0}$ is monotonically increasing with respect to $\tau$ and $\E[|Y|]$ and monotonically decreasing with respect to $\kappa$.
Finally, for $u>0$, define the event 
\begin{align*}
B_{\kappa}(u)^{c}:=\Big\{ \exists v\in [0,u] : |\hat{\phi}(v)- \phi(v)| > \kappa (\log n/n)^{1/2} \Big\}.
\end{align*}
\begin{lemma}  \label{letztes_Hilfslemma}Let $\kappa> \sqrt{1+\frac{2}{K}}$. There exists some $C>0$ which is monotonically increasing with respect to $\E[Y^{2}]$ such that    for any $u\geq 1$ and $n\in \N$, 
\[
 \PP( B_{\kappa}(u)^c  )  \leq   C u  n^{-\frac{1}{K} } . 
\]
\end{lemma}
\begin{proof} Fix $c$ and $\kappa>\tau$, assume that $h=o((n/\log n)^{-1/2} )>n^{-1/2}$. By Corollary \ref{Korollar}, there exists, for any $\tau < \kappa$, some positive integer $n_{0}$, depending on $\E[|Y|]$ and $\tau$ such that for arbitrary $n \geq n_{0}$ we have $B_{c,h,\tau}(u)\subset B_{\kappa}(u)$. It follows that $\PP( B_{\kappa}(u)^{c})\leq \PP(B_{c,h,\tau}(u)^{c})$. The definition of $B_{ c,h,\tau}(u)$ and Lemma \ref{An_Lemma3} give, $\forall n\geq n_{0}$,
\begin{align*}
\PP( B_{ c,h,\tau}(u)^c  ) \leq   \PP(A_{c}^c)  +   2\lceil u /h\rceil 
 n^{-\tau^2/2}\leq   { {\frac{1}{c^2 n} \E[Y^2]   } }+ 2\lceil u /h\rceil n^{-\tau^2/2}.
\end{align*}
With the choice $\tau =\sqrt{1+\frac{2}{K}}$, the statement of the Lemma follows. 
\end{proof}

\subsection{Proof of Theorem \ref{Hauptsatz_Schranke}\label{sec prf1}}
  We have the decomposition $$\|\hat{f}_m-f\|^2=\|f_m-f\|^2+\|\hat{f}_m-f_m\|^2 = \|f_m-f\|^2+\frac{1}{2\pi}\int_{- m}^{ m} |\hat\phi_{X}(u)-\phi_{X}(u)|^2\d u.$$ 

The estimator  $\hat{\phi}_X$ can be rewritten as follows
\begin{align*}
\hat{\phi}_X(u) 
 =  e^{\frac{1}{K}  \hat{\psi}(u) } = e^{\frac{1}{K} \Re \hat{\psi}(u) } 
e^{\frac{1}{K} i \Im \hat{\psi}(u) }=|\hat{\phi}(u)|^{\frac{1}{K}}  e^{\frac{1}{K} i \Im \hat{\psi}(u) },
\end{align*} where $\Re$ and $\Im$ denote the real and imaginary part of a complex number. For nonnegative real numbers, the $K$-th root is unique so that $|\phi(u)|^{1/K}= |\phi_X(u)|$. It follows that 
\begin{align}
|\hat{\phi}_X (u) - \phi_X(u)|^2 
= &  \big| |\hat{\phi}(u)|^{\frac{1}{K}}e^{\frac{i}{K}\Im\hat{\psi}(u) }  - 
|\phi(u)|^{\frac{1}{K} }  e^{\frac{i}{K}\Im{\psi}(u)} \big|^2 \nonumber\\
\label{eq prf UB1} \leq &  2 \big|  |\hat{\phi}(u)|^{\frac{1}{K}} - 
|\phi(u)|^{\frac{1}{K} }  \big|^2 + 2|\phi(u)|^{\frac{2}{K} } \big|
e^{\frac{i}{K}\Im\hat{\psi}(u) }  -  e^{\frac{i}{K}\Im{\psi}(u)}  \big|^2. 
\end{align}

\

Consider the first term of \eqref{eq prf UB1}. Using that $x\rightarrow x^{1/K}$ is H\"older continuous for $x>0$ and the triangle inequality, we derive 
\begin{align*}
  \big|  |\hat{\phi}(u)|^{\frac{1}{K}} - 
|\phi(u)|^{\frac{1}{K} }  \big|^2
\leq &  \big|\hat{\phi}(u) - 
\phi(u)  \big|^{2/K}.
\end{align*} Integrating the former inequality in $u$, taking expectation and applying the Jensen inequality lead to
\begin{align}\label{eq prfthm1}
 \Ebig{\intm    \big|  |\hat{\phi}(u)|^{\frac{1}{K}} - 
|\phi(u)|^{\frac{1}{K} }  \big|^2   \d u  }   \leq &  \intm \big(\E[|\hat{\phi}(u) - 
\phi(u)|^{2}]\big)^{1/K}\d u\leq m n^{-\frac{1}{K}}.\end{align}

In the skewed case, we need to consider the second term of \eqref{eq prf UB1} that is nonzero.  We derive that 
\begin{align*}
 \big|
e^{\frac{i\Im\hat{\psi}(u)}{K} }  -  e^{\frac{i\Im{\psi}(u)}{K} } \big|^2
\leq & \frac{1}{K^2}  \big|\Im\hat{\psi}(u)  -  \Im{\psi}(u)\big|^2 
\leq \frac{1}{K^2}  \bigg|  \intu \frac{\hat{\phi}'(v)}{\hat{\phi}(v) }
- \frac{{\phi}'(v)}{{\phi}(v) } \d  v \bigg|^2.   %\label{eq prf UB31}
\end{align*}

Now,  for arbitrary $v\in \R$, 
\begin{align}
\frac{\dephi(v)}{\ephi(v)}  -\frac{\dphi(v)}{\phi(v)} = &\frac{\phi(v) \dephi(v)-\ephi(v) \dphi(v) }{\phi(v)\ephi(v)}  = \frac{ \phi(v)(\dephi(v) - \dphi(v) )  - \dphi(v) (\ephi(v) - \phi(v)        ) }{\phi(v)\ephi(v)} \nonumber \\
=&\dfrac{ (\dephi(v) - \dphi(v) )/\phi(v)  - (\dphi(v)/\phi(v)^2) (\ephi(v) - \phi(v)        ) } {\ephi(v)/\phi(v)} \nonumber  \\\label{eq:der}=& \dfrac{   \Big(- \frac{(\ephi(v)-\phi(v))}{\phi(v)}   \Big)' }{\Big(1- \frac{(\ephi(v)-\phi(v))}{\phi(v)}    \Big)   }.
\end{align}

 First, we consider the case where $m$ lies in $[0, u_{n}^{(\gamma,\epsilon)}]$, where $u_{n}^{(\gamma,\epsilon)}$ is defined in \eqref{eq uek}. Set $\kappa=(1+\eps)\gamma$, by Corollary \ref{Korollar} and by the definition of  the event  $B_{\kappa} (m)$ and of $u_{n}^{(\gamma,\epsilon)}$, we find that on $B_{\kappa}(m)$, 
\begin{equation}  \label{Bm}
\forall u\in [-m, m ]   :  |\hat{\phi} (u) -  \phi(u) | \leq \kappa \sqrt{\log n} n^{-\frac{1}{2} } \leq 1/(1+\eps)  |\phi(u) |. 
\end{equation}
Then, on the event $B_{\kappa}(m)$,  a  Neumann series expansion, along with formula \eqref{eq:der} and formula $\eqref{Bm}$ gives for $v\in[-m,m]$, 
\begin{align*}
  \frac{\dephi(v)}{\ephi(v)}  -\frac{\dphi(v)}{\phi(v)}  =- \sum_{\ell=0}^{\infty} \Big( \frac{\ephi(v)-\phi(v)}{\phi(v)}   \Big)' \Big( \frac{\ephi(v)-\phi(v)}{\phi(v)}    \Big)^{\ell}.
\end{align*}  
The following representation holds
\begin{align*}   \Big( \frac{\ephi(v)-\phi(v)}{\phi(v)}   \Big)' \Big( \frac{\ephi(v)-\phi(v)}{\phi(v)}    \Big)^{\ell} = \frac{1}{\ell +1} \Big[ \Big( \frac{\ephi(v)-\phi(v)}{\phi(v)}    \Big)^{\ell+1}\Big]'.
\end{align*}
Moreover, we have $\hat{\phi}(0)- \phi(0)=0$ and get 
\begin{align*}
 \mathds{1}_{B_{\kappa}(m)}\Big| \intu  \Big( \frac{\dephi(v)}{\ephi(v)}  &-\frac{\dphi(v)}{\phi(v)}   \Big) \d v  \Big|
\leq \sum_{\ell=0}^{\infty} \frac{1}{\ell+1}   \frac{|\hat{\phi}(u) - \phi(u) |^{\ell +1}  }{|\phi(u)|^{\ell+1}  } \\
 \leq  & \frac{|\hat{\phi}(u) - \phi(u) |}{|\phi(u)|  }  \sum_{\ell=0}^{\infty} \frac{(1+\eps)^{-\ell} }{\ell+1}  = (1+\eps) \log(1+1/\eps)\frac{|\hat{\phi}(u) - \phi(u) |}{|\phi(u)|  } .
\end{align*}
Gathering all the terms together, we have shown that for $m\in[0, u_{n}^{(\gamma,\epsilon)}]$, 
\begin{align}
 \Ebig{  \mathds{1}_{{B_{\kappa}(m)}}  & \intm   |\phi_X(u)|^2 \big|e^{\frac{i}{K} \Im  \hat{\psi}(u)}-
e^{\frac{i}{K}  \psi(u) }  \big|^2   \d u }\nonumber  \\
\leq  &     \frac{c(\eps) }{K^2} \intm    |\phi_X(u)|^2 \frac{\ebig{|\hat{\phi}(u) -\phi(u)|^2} }{|\phi(u)|^2}  \d u  
 \leq   \frac{ c(\eps)  }{ K^2} \intm  \frac{n^{-1}}{|\phi_X(u)|^{2(K-1)} }  \d u ,\label{eq prfthm3}
\end{align}
where $c(\eps):= (1+\eps) \log(1+1/\eps)$. 
On $B_{\kappa}(m)^{c}$, we use the majorant $|e^{i\Im\psi/K}-e^{i\Im \hat{\Psi}/K}|\leq 2$ and Lemma \ref{letztes_Hilfslemma}
\begin{align}
  \Ebig{  \mathds{1}_{{B_{\kappa}(m)^c}}   \int_{- m}^{ m}   |\phi_X(u)|^2 \big|e^{\frac{i}{K} \Im  \hat{\psi}(u)}-
e^{\frac{i}{K}  \psi(u) }  \big|^2   \d u }  \leq & 4 \PP( B_{\kappa}(m)^c ) \int_{- m}^{  m}   |\phi_X(u)|^{2}  \d u \nonumber \\
  \leq  & 4 \|\phi_X\|^2  m n^{-\frac{1}{K} }.\label{eq prfthm2} 
\end{align}

\

Secondly, we consider the case $m\geq u_{n}^{(\gamma,\epsilon)}$. The series of inequalities 
\begin{align}  
  |\hat{\phi}_X(u) -\phi_X(u) |  \leq & 2 |\phi_X(u)|  +||\hat{\phi}_X(u)| - |\phi_X(u)||  \nonumber\\
\label{Abs_3105}
= & 2 |\phi_X(u)|  +||\hat{\phi}(u)|^{\frac{1}{K}} - |\phi(u)|^{\frac{1}{K} }|  
\leq  2 |\phi_X(u)|  +|\hat{\phi}(u) -\phi(u)|^{\frac{1}{K} },
\end{align}
combined with the Jensen inequality, implies that   
\begin{align}
 \int_{ |u| \in [u_{n}^{(\gamma,\epsilon)}, m] }\limits \hspace*{-0.5cm} \Ebig{|\hat{\phi}_X(u) -\phi_X(u) |^2  }  \d u  \leq&  4 \hspace*{-0.5cm}\int_{  |u|\in [u_{n}^{(\gamma,\epsilon)}, m] }\limits \hspace*{-0.5cm} |\phi_X(u) |^2 \d u +  4 \hspace*{-0.5cm}\int_{  |u|\in [u_{n}^{(\gamma,\epsilon)}, m] }\limits \hspace*{-0.5cm} \Ebig{ |\hat{\phi}(u) - \phi(u) |^2 }^{\frac{1}{K} } \d u \nonumber \\
\leq &  4 \hspace*{-0.5cm} \int_{|u|\in [u_{n}^{(\gamma,\epsilon)}, m] } \limits  \hspace*{-0.5cm}|\phi_X(u) |^2 \d u   +   8 n^{-\frac{1}{K } } m.  \label{eq jens}
\end{align}
Gathering equations \eqref{eq prfthm1}, \eqref{eq prfthm3}, \eqref{eq prfthm2} and \eqref{eq jens} and defining $C_{1}=c(\eps)$ and $C_{2}=(9+4\|\phi_{X}\|^{2})$, completes the proof of of Theorem \ref{Hauptsatz_Schranke}. 
\hfill $\Box$

\subsection{Rates of convergence: Proof of Proposition \ref{prop}.} Proposition \ref{prop} is a direct consequence of Theorem \ref{Hauptsatz_Schranke}.  To establish the result we bound independently each term of the right hand side of \eqref{equation_upper}.
\begin{itemize}
\item[(i)]    Fix  $\beta, b, C,C',\gamma$  and $\eps$.  There exist a constant $C_1$ which depends  on $\gamma,\eps$  and decreases with respect to  $C$  and   a constant $C_2$  depending on $a$ and decreasing with respect to  $C'$  such that  for any  $f\in \mathcal{F}(\beta, b, C, C', C_X, C_f)$ we have
\begin{align*}
u_{n}^{(\gamma, \eps)}  \geq C_1(n/\log n)^{\frac{1}{2b K}  }
 \quad \text{and}   \quad   u_n^{(a)}   \leq C_2  n^{\frac{1}{2\beta K}  }.
\end{align*}
Using the fact that $m^*\leq u_n^{(a)}$  and the definition of $u_n^{(a)}$,  we can estimate the second term in \eqref{equation_upper} as follows
\begin{align*}
\frac{1}{n}\int_{-m^*}^{m^*}  \frac{1}{|\phi_X(u)|^{2(K-1)} }  \d u
\leq   & \frac{ 2u_n^{(a)} }{n}  |\phi_X(u_n^{(a)})|^{-2(K-1) }   
\leq  2 C_2 n^{\frac{1}{2\beta K} -1 }  a n^{\frac{2(K-1)     }{2K    }   }  \\
=& 2a C_2 n^{-\frac{2\beta -1}{2\beta K} }  \leq 2a C_2 n^{-\frac{2\beta -1}{2b K} } .
\end{align*}
The first and fourth term appearing in Theorem 1 can be controlled using that 
\begin{align*}
\int_{ \{|u| \geq u_n^{(\gamma, \eps)}  \} }    |\phi_X(u)|^2 \d u
\leq  \frac{2 C'}{2\beta -1}  \big( u_n^{(\gamma, \eps)} \big)^{-(2\beta -1)}
\leq\frac{2 C'}{2\beta -1} \Big( C_1 \frac{n}{\log n}  \Big)^{-\frac{2\beta -1}{2b K}}.
\end{align*}
Finally, the third term in \eqref{equation_upper} is bounded by
\begin{align*}
n^{-\frac{1}{K}  } m^*   \leq  C_2  n^{-\frac{2\beta - 1}{2\beta K}   } 
\leq C_2 n^{-\frac{2\beta -1}{2bK}  }. 
\end{align*}
We have thus shown that for a positive constant $\mathcal{C}$n which depends on the choice of $\kappa$ and $\eps$, which increases with respect to $C$, $C_X$ and $C_f$ and which decreases with respect to 
$\beta$ and $C'$,  we have that
\[
\E[\| f- \hat{f}_{m^*}  \|^2 ]  \leq\mathcal{C}  n^{-\frac{2\beta -1}{2bK}  }. 
\]
\item[(ii)] Fix $\beta, \rho, b, c, C,C',\gamma$ and $\eps$. Tere exists a constant $C_1$ depending on  $\gamma,\eps$   and decreasing with $C, c, b$
and a constant $C_2$ depending on  $a$  decreasing with respect to $C'$   such that  for any $f\in \mathcal{H}(\beta, \rho, b, c, C, C', C_X, C_f)$, we have
\begin{align*}
u_n^{(\gamma, \eps)}   \geq C_1 (\log n)^{\frac{1}{\rho} }
\quad \text{and}   \quad     u_n^{(a)}   \leq C_2 n^{\frac{1}{2\beta K}  }.
\end{align*}
Again, we may use the definition of $m^*$ and the estimate for $u_n^{(a)}$ to derive that 
\begin{align*}
\frac{1}{n}\int_{-m^*}^{m^*}   \frac{1}{|\phi_X(u)|^{2(K-1)}  }  \d u
\leq  2 a C_2  n^{-\frac{2\beta -1}{2\beta K}  }  =o \big((\log n)^{-\frac{2\beta - 1}{\rho}  }\big).
\end{align*}
Moreover, we have the upper bounds
\begin{align*}
\int_{ \{|u| \geq u_n^{(\gamma, \eps)}  \} }    |\phi_X(u)|^2 \d u
\leq    \frac{ 2  C'}{2\beta -1 } (C_1 \log n)^{-\frac{2\beta -1}{\rho} }
\end{align*}
and
\begin{align*}
m^*  n^{-\frac{1}{K}  }  \leq    C_2  n^{-\frac{2\beta -1}{2\beta K}  } =o\big( (\log n)^{-\frac{2\beta -1}{\rho}  }\big).  
\end{align*}
Putting the above together, we find that 
\[
\E[\| f-  f_{m^*} \|^2 ]  \leq \mathcal{C}  (\log n)^{-\frac{2\beta-1}{\rho} },
\]
for some $\mathcal{C}\geq 0$ which is  increasing with respect to $C, c, b, C_X, C_f$  and decreasing with respect to $\beta$ and $C'$. 
\item[(iii)] Fix $\rho, c, C,C',\gamma$ and $\eps$. There exists a constant $C_1$ depending on $\gamma$ and $\eps$  which is decreasing with respect to  $c$ and $C$  and a constant $C_2$ depending on $\gamma$ and $\eps$ which is decreasing with respect to $c$ and $C'$ 
 such that for any $f\in \mathcal{G}(\rho, c, C, C', C_X, C_f)$, we have
\begin{align*}
u_{n}^{(\gamma, \eps)}  \geq C_1 (\log n)^{\frac{1}{\rho} }  \quad \text{and}  \quad
u_n^{(a)} \leq C_2 (\log n)^{\frac{1}{\rho} }.
\end{align*}
In analogy with (i) and (ii), we have
\begin{align*}
\frac{1}{n}\int_{-m^*}^{m^*}   \frac{1}{|\phi_X(u)|^{2(K-1)}  }  \d u
\leq 2 a C_2 n^{-\frac{1}{K} } (\log n)^{\frac{1}{\rho} }.
\end{align*}
On the other hand, for some constant $d$ depending on $c$, $C$ and $C'$, we derive from \eqref{eq G} that
\begin{align*}
\int_{ \{|u| \geq u_n^{(\gamma, \eps)}  \} }  |\phi_X(u)|^2 \d u \leq & \frac{2C'}{c}   \exp( -2 c \big( u_n^{(\gamma, \eps)}\big)^{\rho}  )  \leq   d \big(u_n^{(\gamma, \eps)}\big)^{1-\rho} \big|\phi_X\big( u_n^{(\gamma,\eps)}
\big)\big|^2  \\
\leq &  d C_2 (\log n)^{ \frac{1}{\rho} - 1}  \Big( \frac{n}{\log n}  \Big)^{- \frac{1}{K} }
\leq  d C_2 n^{-\frac{1}{K}  } (\log n)^{\frac{1}{\rho} }.
\end{align*}
We also get that
\begin{align*}
n^{-\frac{1}{K} } m^*  \leq  C_2  n^{-\frac{1}{K} } (\log n)^{\frac{1}{\rho} }.
\end{align*}
This  implies for a constant $\mathcal{C}$ which depends on $C,C'$ and $c$, 
\[
\E[\| f-  f_{m^*} \|^2 ]  \leq \mathcal{C}  (\log n)^{-\frac{2\beta-1}{\rho} },
\]
which completes the proof of Proposition \ref{prop}.   \hfill $\Box$
\end{itemize}
\textbf{Remark.} The dependence of $\mathcal{C}$ is not easy to express in a closed form.  However, $\mathcal{C}$ is continuous in the constants involved.

\subsection{Lower bound: Proof of Theorem \ref{prop_2}}
Following Theorem 2.7 in Tsybakov (2009), the lower bound is established by considering a decision problem between  an increasing number of competing  densities $f_0, \ldots, f_{{M_n}}$ contained in  $\mathcal{F}(\beta, b, C, C',C_X, C_f)$, for some $M_{n}\geq 1,$ such that the following two conditions are satisfied
\begin{itemize}
\item[(i)]   $\| f_{j}- f_{k}\|^2 \geq d n^{-\frac{2\beta - 1}{2bK} },\   \forall \, 0\leq j<k\leq M_n$    for some positive constant $d$,
\item[(ii)]  $\frac{1}{M_n} \sum_{j=1}^{M_n}\limits  \KL(\PP^{\otimes n}_{j,K}, \PP^{\otimes n}_{0,K} )  \leq \alpha  \log M_n,\  \forall 0\leq j \leq M_n$ and for some $0<\alpha<1/8$,
\end{itemize} where
$\PP_{j,K}$ denotes the probability measure corresponding to the density $f_{j}^{\ast K}$  and 
$\KL$ is the Kullback-Leibler divergence.  In the sequel, when there is no risk of confusion,  $\mathcal{F}$ stands for the class $\mathcal{F}(\beta, b, C, C', C_X, C_f)$. 

We start by constructing a density contained in $\mathcal{F}$, with characteristic function $\varphi$, such that
\begin{align} \label{auf_und_ab}
\liminf_{|u|\to  \infty}  |\phi(u)| |u|^{b} < \infty  \quad \text{and}  \quad 
\limsup_{|u|\to \infty}  |\phi(u)| |u|^{\beta }>0. 
\end{align}

We start by considering the density 
\begin{align*}
g_0(x)=\dfrac{3 \sin^4(x)}{2 \pi x^4}.
\end{align*}
The  characteristic function $\phi_{g,0}$ of $g_0$ is, up to a multiplicative  factor 3/2,  the  $4$-fold auto-convolution of the distribution $\mathcal{U}([-1,1])$ and hence supported on $[-4,4]$. 

To ensure the second part of  \eqref{auf_und_ab}, we define for some positive constant $c$, 
\begin{align*}
g(x) = g_0(x) + c \sum_{k=1}^{\infty} (1+(8k))^{-\beta} \cos (12kx)   g_0(x).
\end{align*}
Using the boundedness of the cosine-function, along with $\beta>1$ and an appropriate choice of $c$, $g$ is nonnegative. Moreover, since $g$ integrates to 1, $g$ is a density.  The multiplication with $\cos( 12 kx)$ corresponds to a left and right shift in the Fourier domain, so the characteristic function of $g$ is 
\begin{align}
\phi_g(u)= \phi_{g,0}(u)+\frac{c}{2}\sum_{k=1}^{\infty}  (1+(8k))^{-\beta}  ( \phi_{g,0}(u+12k)  + \phi_{g,0}(u-12k)  )  \label{US_1}.
\end{align}

For $g$ constructed as above, we can a priori not guarantee that \eqref{Konst_Schranke} is valid. However, this can be ensured with an  appropriate rescaling. 

\

So far, $g$ is not contained in the class $\mathcal{F}$, since the absolute value of  $\phi_{g}$ is not bounded from below, it does even have zeros.  To ensure the lower bound on the characteristic function under consideration, we proceed as follows.  Let $h$ be a symmetric  density which decays, in the time domain, as $x^{-4}$ and  whose characteristic function decays as  $|u|^{-b}$. For example, $h$ can be selected as the mixture of a density of the form $c'( 1+x^2)^{-2}$ and of 
a symmetric bilateral $\Gamma(b/2, \lambda)$ density.

Set
\[
f_0(x) = c_1 g(x)  + c_2 h(x) .  
\]
For  a positive constants $c_1$  to be chosen and $c_2=1-c_1$. 
Finally, by construction  $f_0$ belongs to the class  $\mathcal{F}$. 

\

A collection of competing  densities $f_0, \ldots, f_{{M_n}}$ is  obtained by considering  perturbed versions of the characteristic function $\phi_0$ of  $f_0$. Let  $\psi$ be  a complex-valued function satisfying the following properties:
\begin{itemize}
\item $\psi$ is supported on $[-8,8]$.
\item  $\psi(u)= 1-e^{i2\pi/K},\ \forall u\in[-4,4].$  
\item $\psi$ has three continuous (complex) derivatives.
\item $|1-\psi(u)| \geq 1/2\    \forall u\in [-8,8]$. 
\item $ |\psi(u)| \leq 2\    \forall u\in [-8,8]$. 
\end{itemize}
For $n\in \N$,  a family of perturbations is defined as follows.  Let  $L_n= 12k_n$,  with some $k_n\in \N$  to be specified.  Let $m_n$ be a positive integer to be chosen.  For $1\leq m \leq m_n$ define 
\[
\psi_m(u) =  \psi( u- L_n -12 m) +\overline{\psi} (u+L_n+12 m).
\]
Given any subset  $\mathcal{M} \subseteq \{1, \ldots, m_n\}$,  a perturbed version of $\phi_0$ is defined by  
\[
\phi_{\mathcal{M}}(u):= \phi_{0}(u) - \sum_{m\in \mathcal{M}} \psi_m(u)\phi_0(u).
\]
In the sequel, define the sets
\[
I_{n,m}: =[-L_n - 12  m - 4, -L_n -12 m+4]  \cup  [L_n + 12  m - 4, L_n + 12 m+4],
\] 
and 
\[
J_{n,m}:=\big( [-L_n - 12  m - 8, -L_n -12 m+8]  \cup  [L_n + 12  m - 8, L_n + 12 m+8] \big) \setminus 
I_{n,m},
\]which are disjoints for distinct values of $m$. Note that the following holds, for all $m\in \mathcal{M}$.  The function $\psi_m$ is compactly supported on the interval $I_{n,m}\cup J_{n,m}$ and if $u\in I_{n,m}$ we have $\psi_{m}(u)=(1-e^{i2\pi/K})$. 

Now, if we consider  $ \phi_{g}(u) -\sum_{m\in \mathcal{M} }  \psi_m(u) \phi_{g}(u)  $, by construction, this function is equal to the function $\phi_{g}e^{i2\pi/K}$ on any interval of the form $I_{n,m}, \ m\in \mathcal{M}$ and with $\phi_{g}$ elsewhere since on the interval $J_{n,m}$, the function $\phi_{g}$ is zero.  Using the same arguments as for formula \eqref{US_1}, we find that the inverse Fourier transform of this quantity is non negative and hence is a density.

Next, consider $\phi_h(u)- \sum_{m\in \mathcal{M} }  \psi_m(u) \phi_{h}(u)  $.  By construction of $\psi$, the inverse Fourier transform of $\psi_m$ decays faster than $h$ and the inverse Fourier transform of $\psi_m \phi_h$ is bounded by $c L_m^{-b} |\Fourier^{-1}\psi_m |$. It follows that the inverse Fourier transform of   $\phi_h(u)- \sum_{m\in \mathcal{M} }  \psi_m(u) \phi_{h}(u)  $ is non-negative, provided that $m_n\leq L_m/c$. Finally,  by the definition of $\psi$,  $\phi_{\mathcal{M}}$ has the same asymptotic behavior as $\phi_0$ for any subset $\mathcal{M}$. Finally, $\phi_{\mathcal{M} }$  belongs to $\mathcal{F}$.

\

We are now ready to check that (i) holds  by bounding from below the $L_{2}$-distance between the competing densities.   By definition of the $\psi_m$, given two sets of indices $\mathcal{M}_1 , \mathcal{M}_2  \subseteq \{m_1, \ldots, m_n\}$,  we have
\begin{align*}
\big|\phi_{\mathcal{M}_1 }(u)|_{{I_{n,m}}}- \phi_{\mathcal{M}_2}(u)|_{{I_{n,m}}} \big|  =&\mathds{1}_{ \mathcal{M}_1\triangle\mathcal{M}_2 }(m) |1-e^{i2\pi/K}|\big|\phi_0(u)|_{{I_{n,m}}} \big|\\
 = &  \mathds{1}_{ \mathcal{M}_1\triangle\mathcal{M}_2 }(m)  2|\sin(\pi/K)| \big|\phi_0(u)|_{I_{n,m}}\big|,
\end{align*} 
with $\triangle$ denoting the symmetric difference.  Using this, along with formula \eqref{US_1}, we derive that there exists a positive constant $C$ such that
\begin{align}\label{L-Abstand}
 \|\phi_{\mathcal{M}_1 } -   \phi_{\mathcal{M}_2 }  \|^2  \geq C |\mathcal{M}_1\triangle\mathcal{M}_2|   (1+|L_m+12m_n|)^{-2\beta} \int_{-4}^{4} |\phi_{g,0}(u)|^2 \d u.
\end{align}  
The Varshamov-Gilbert bound (see, e.g. Lemma 2.9 in Tsybakov (2009)) guarantees that for  $m_n\geq 8$, there exist  a family 
$\{\mathcal{M}_{n,k}: 0\leq k\leq  M_n\}$ of subsets of $\{1,\ldots, m_n\}$ such that
\begin{itemize}
\item  $\mathcal{M}_{n,0}=\emptyset $,  
\item  $M_n\geq 2^{m_n/8}$ and
\item $|\mathcal{M}_{n,k} \triangle \mathcal{M}_{n,j} |\geq  m_n/8, \   \forall\,  0\leq j < k \leq M_n$. 
\end{itemize}
Combining this with formula \eqref{L-Abstand}, we find that for $n\in \N$ large enough, there exist $M_n=2^{m_n/8}$ densities  $f_0, \ldots, f_{{M_n}}$  belonging to $\mathcal{F}$ and for which 
\[
\|f_{j}- f_{k} \|^2 \geq C m_n  m_n^{-2\beta},\  j\not = k 
\]
holds, provided that $L_n\asymp m_n$. 

\

Let us now bound the Kullback-Leibler divergence to establish (ii). Using that the Kullback-Leibler divergence is bounded by the $\chi^2$-distance, as well as the fact that 
\[
\KL(\PP^{\otimes n}, \QQ^{\otimes n} ) =n \KL(\PP, \QQ), 
\]
it is enough to consider the $\chi^2$-distance between the competing densities $f_{j}^{\ast K}, \ j=0,\ldots, M_n$,
\begin{align*}
\chi^2(\PP_{j,K}, \PP_{0,K} )=\int_{-\infty}^{\infty}  \frac{ (f_j^{\ast K} (x)- f_0^{\ast K}(x) )^2}{f_0^{\ast K}(x) }   \d x. 
\end{align*}
Since $f_{0}^{\ast K}$ and $f_{0}$ decay at the same rate $x^{-4}$, we have a some positive constant $C$,  
\begin{align*}
\chi^2( \PP_{j,K}, \PP_{0,K} )
\leq & C \int_{-\infty}^{\infty} (1+x^4)(f_j^{\ast K} (x)- f_0^{\ast K}(x) )^2  \d x  \\
= & C \int_{-\infty}^{\infty} (f_j^{\ast K} (x)- f_0^{\ast K}(x) )^2  \d x   + C \int_{-\infty}^{\infty} (x^2f_j^{\ast K} (x)- x^2 f_0^{\ast K}(x) )^2  \d x . 
\end{align*}
The multiplication with $x^k$ corresponds, up to the factor $i^{\m k}$, to taking the $k$-th derivative in the Fourier domain. The Plancherel formula gives\begin{align*}
\chi^2( \PP_{j,K}, \PP_{0,K})\leq &  \frac{C}{2\pi} \int_{-\infty}^{\infty}|\phi_j^{K}(u) - \phi_0^{K}(u) |^2 \d u   + \frac{C}{2\pi}\int_{-\infty}^{\infty}  |(\phi^{K}_j)''(u)  - (\phi_0^{K})''(u)|^2 \d u.  
\end{align*}
By construction,  on the intervals  $I_{n,k},\ k\in \mathcal{M}_n$,  $\phi_0$ is strictly positive, we  have
\[
\phi_j(u) = \phi_0(u) e^{i2\pi/K}=|\phi_0(u)| e^{i2\pi/K}.
\]
Consequently,  $\phi_j^{K}$ and $\phi_0^{K}$ are equal on $I_{n,k}$. For $u\in J_{n,m}, \ m\in \mathcal{M}_n$, 
\[
(\phi_j^{K}(u) - \phi_0^{K}(u)  ) =   \phi_0(x)^{K}  (1-(1-\psi_m)^{K}(u) ).   
\]
Now, on the intervals  $J_{n,m}$,  $\phi_0$ agrees with  $\phi_{h}$.   As $h$  is the mixture of a density in $C(1+x^2/c^2)^{-2}$ and a bilateral gamma density, we can derive  from the analysis given in K\"uchler and Tappe (2008) that  $|(\phi_{0}^K)^{(\ell)}(u)|\lesssim |u|^{-Kb-\ell},\ \ell=0,1,2$.  Along with the boundedness of the derivatives of $\psi(u)$, this implies  that for $u\in J_{n,m}$,  
\begin{align*}
|\phi^K_j (u) - \phi_0^K|+ |(\phi^{K}_j)''(u)  - (\phi_0^{K})''(u)|^2 \lesssim |u|^{-2bK}.
\end{align*}
Finally, $\phi_j^{K}=\phi_0^K$ on the complement of the set $\cup_{m\in \mathcal{M}_n}
(I_{n,m} \cup{J}_{n,m} )$. 
It follows that 
\[
\chi^2(\PP_{j,K}, \PP_{0,K} ) \lesssim m_n L_n^{-2 bK} ,
\]
and
\begin{align*}
\frac{1}{M_n} \sum_{k=1}^{M_n} \KL(\PP_{j,K}^{\otimes n}, \PP_{0,K}^{\otimes n} ) \lesssim  n m_n L_n^{-2bK}\leq (1/8 \log M_n )  n L_n^{- 2bK }. 
\end{align*}
The choice  $L_n\asymp m_n \asymp n^{\frac{1}{2bK} }$ and Theorem 2.7 in Tsybakov (2009) imply the result. 
 \hfill $\Box$
\subsection{Adaptive estimation: Proof of Theorem \ref{Hauptsatz_ad}.}
In the sequel, we write  $\hat{m}$ instead of  $\hat{m}_{\eta}$. Recall that, for some $\eta>1$, $$\hat{m}=\min\big\{\min\{u>0: |\hat{\phi}(u) | = (nK)^{-\frac{1}{2} }  +\sqrt{\frac{\eta }{K} \log n}  
n^{-\frac{1}{2} } \}  ,  n^{\frac{1}{K}  }  \big\}.$$ 

 In what follows, $\mathcal{E}$ stands for one of the function classes, namely,  $\mathcal{F}(\beta, b, C,C',C_X, C_f)$, $\mathcal{H}(\beta, \rho, b, c, C,C', C_X, C_f) $ or $\mathcal{G}(\rho, c, C, C',C_X, C_f)$.  Depending on the class considered, 
 let $r_n=~(n/\log n)^{\frac{2\beta -1}{2bK}}$, $r_n= (\log n)^{\frac{2\beta -1}{\rho} }$ or 
 $r_n=  (\log n)^{\frac{1}{\rho} } n^{-\frac{1}{K} }$.

Let $f$ be any density in $\mathcal{E}$ and define 
\[
m_0:= \min \{u>0: |{\phi}(u)|  = (2 \sqrt{\eta /K } +\kappa ) (\log(n)/ n)^{1/2}   \}
\]
and 
\[
  m_1:= \min \{u>0: |{\phi}(u)|  = (nK)^{-\frac{1}{2} }   \}.
\]

Firstly, using the definition of $m_0$, along with Lemma \ref{letztes_Hilfslemma} and the triangle inequality,  we find that 
\begin{align}
\label{eq m0}
&\PP\big(   \hat{m}< m_0  \big)   
\leq  \PP\big(    \exists u\in [0, m_0] : |\hat{\phi}(u) - \phi(u)| \geq \kappa \sqrt{\log n} n^{-\frac{1}{2} }  \big)  \leq  C m_0 n^{-\frac{1}{K}  },\end{align}
with some constant $C$ depending on $\E[X^{2}]< \overline{C}_X$. 

Secondly, by the Hoeffding inequality, the triangle inequality and by the continuity of the empirical characteristic function, 
\begin{align}
 \PP\big( \hat{m}>m_1\big)  \leq &\PP \big( |\hat{\phi}(m_1)|  > (nK)^{-\frac{1}{2} }
+\sqrt{\eta\log(n)/K} n^{-\frac{1}{2} }   \big)\nonumber \\
\leq  &  \PP\big( |\phi(m_1) - \hat{\phi}(m_1) |
> \sqrt{\eta\log(n) /K} n^{-\frac{1}{2} }   \big)  \leq n^{-\frac{\eta}{K} }. \label{eq m1}
\end{align}

The proof of Theorem \ref{Hauptsatz_Schranke} implies that 
\begin{align} 
  \E[&  \| f- \hat{f}_ {\hat{m}} \|^2  \mathds{1}_{\{\hat m\in[m_{0},m_{1}] \} } ]  \nonumber\\
 \label{ad_11} \leq &  \| f- f_{{m_0}} \|^2   +
\frac{C_1}{n K^2} \int_{-  {m_1}  }^{ m_1 }\limits 
\frac{1}{|\phi_X(u)|^{2(K-1)}  }  \d u +  C_2 m_1 n^{-\frac{1}{K} }  +  \frac{2}{\pi} \int_{ [u_{n}^{(\kappa,\epsilon)}, m_1] }\limits |\phi_X(u) |^2 \d u,
\end{align}
 with a constant $C_1$ depending on the definition of $\eps$ and with $C_2 $ depending on $\|f\|<\overline{C}_f$ and on $\E[X^{2} ] <\overline{C}_X$. 
We can now use the same arguments as in the proof of Proposition 2 to derive from \eqref{ad_11} that 
\begin{align} \label{Orakelgleichung_1}
\hspace{-0.3cm} \E[\|f -& \hat{f}_{\hat{m}}  \|^2 \mathds{1}_{ \{\hat m\in[m_{0},m_{1}] \} } ]    \leq \mathcal{C}_1  r_n
\end{align}
for a positive constant $\mathcal{C}_1$ depending on $I$
and on the choice of the parameters. 

Consider now the exceptional set $ {\hat{m}\notin[m_{0},m_{1}] }$. Applying the Plancherel formula  and arguing along the same lines as in the proof of \mbox{Theorem \ref{Hauptsatz_Schranke}}, we derive that 
\begin{align} 
 & \E[  \| f- \hat{f}_{\hat{m } }\|^2 \mathds{1}_{ \{\hat{m}<m_0 \} }]  \leq   \PP(\hat{m}<m_0) \| f\|^2  +  \frac{1}{2\pi}\int_{ - m_0}^{m_0}    \E[ | \hat{\phi}_X (u) - \phi_X(u)|^2  ] \d u\nonumber\\
\leq  &  \PP(\hat{m}<m_0) \| f\|^2   +  \frac{C_1}{nK^2 } \int_{- m_0 }^{ m_0 }   \frac{1}{|\phi_X(u)|^2 } \d u  + C_2 n^{-\frac{1}{K} } m_0 +\frac{2}{\pi} \hspace*{-0.5cm} \int_{ [ u_{n}^{(\gamma,\epsilon)}, m_0] }\limits  \hspace*{-0.5cm} |\phi_X(u) |^2 \d u .  \label{Ausnahmemenge_1}
\end{align}
Finally,  the fact that $\hat{m}\leq n^{\frac{1}{K}}$, along with the Plancherel formula and  formula \eqref{Abs_3105}  leads to
\begin{align*} 
 \E[  \| f&- \hat{f}_{\hat{m } }\|^2  \mathds{1}_{ \{ \hat{m}>m_1 \} }  ]\\
 \leq  &     \| f- f_{{m_1} }\|^2   +4  \PP(\hat{m}>m_1)   \int_{-  n^{\frac{1}{K}} }^{ n^{\frac{1}{K}} } | \phi_X(u)|^2 \d u 
+  \int_{-n^{\frac{1}{K}}}^{n^{\frac{1}{K}} } \E[\mathds{1}_{ \{ \hat{m}>m_1 \} }  |\hat{\phi}(u)- \phi(u)|^{\frac{1}{K} } ] \d u.
\end{align*}
It follows, using successively the H\"older, the Jensen and the Rosenthal inequalities, that for some constant $C$ and some $C'$ depending on $\eta$,\begin{align}
\label{Ausnahmemenge_2}\hspace{-0.3cm}\E[  \| f- \hat{f}_{\hat{m } }\|^2  \mathds{1}_{ \{ \hat{m}>m_1 \} }  ] \leq  \|f- f_{ {m_1 }}\|^2 
+ Cn^{-\frac{1}{K} }  \| \phi_X\|^2   + C' \PP(  \hat{m} >m_1)^{  \frac{1}{\eta}    }\int_{- n^{\frac{1}{K}} }^{ n^{\frac{1}{K}}}   n^{-\frac{1}{K} }\d u.
\end{align}
Combining \eqref{Ausnahmemenge_1} and  \eqref{Ausnahmemenge_2} together with \eqref{eq m0} and \eqref{eq m1} and plugging in $m_{0}$ and $m_{1}$, we have shown that for  a positive constant $\mathcal{C}_2$ depending again on $I$ and  the choice of $\eps,\delta$ and $\eta$, 
\begin{align} \label{Orakelgleichung_2}
\hspace{-0.3cm} \E[\|f -& \hat{f}_{\hat{m}}  \|^2 \mathds{1}_{\hat m\notin[m_{0},m_{1}] } ]    \leq \mathcal{C}_2  r_n.
\end{align}
Putting  \eqref{Orakelgleichung_1} and \eqref{Orakelgleichung_2} together, we have shown that  
there exists a constant  $\mathcal{C}$ depen\-ding on the choice of $\eta,\delta $ and $\eps$ and on  $I$, such that for any $f\in \mathcal{E}$,
\begin{align*}
\frac{ \E[ \| f- \hat{f}_{\hat{m}} \|^2 ] }{  r_n } \leq \mathcal{C}.
\end{align*}
Taking the suprema   gives the statement of the theorem.   \hfill $\Box$

\end{document}